\documentclass[a4paper,11pt]{amsart}
\usepackage{ulem}
\usepackage{latexsym}
\usepackage{amssymb}
\usepackage{graphicx}
\usepackage{wrapfig}
\usepackage{amsthm}
\usepackage{float}
\usepackage{epsfig}
\usepackage{multirow}
\usepackage{caption}
\usepackage[toc,page]{appendix}
\usepackage{hyperref}
\RequirePackage{color}
\usepackage{amscd,enumerate}
\usepackage{amsfonts}
\usepackage{MnSymbol}
\usepackage{verbatim}
\usepackage[english]{babel}
\input xy
\xyoption{all}
\usepackage{mathrsfs}

\usepackage{multicol}
\usepackage{eucal}

\hypersetup{ colorlinks,
linkcolor=blue,
filecolor=green,
urlcolor=blue,
citecolor=blue }

\input xygraph

\usepackage{calrsfs}
\DeclareMathAlphabet{\pazocal}{OMS}{zplm}{m}{n}

\newcommand{\s}{\sigma}
\def\esc#1{\langle #1\rangle}
\def\2{{\bf 2}}
\def\1{{\bf 1}}
\def\remove#1{}
\newcommand{\K}{\mathbb{K}}

\newcommand{\N}{\mathbb{N}}

\newcommand{\spa}{{\rm span}}

\newcommand{\ann}{{\rm ann}}

\newcommand{\car}{{\rm char}}

\newcommand{\ran}{{\rm rank}}
\newcommand{\ex}{{\rm Ex}}

\newcommand{\A}{\mathsf{A}}

\newcommand{\B}{\mathcal{B}^{(n\times n)}}
\newtheorem{lemma}{Lemma}[section]
\newtheorem{corollary}[lemma]{Corollary}
\newtheorem{theorem}[lemma]{Theorem}
\newtheorem{proposition}[lemma]{Proposition}
\newtheorem{remark}[lemma]{Remark}
\newtheorem{definition}[lemma]{Definition}

\newtheorem{example}[lemma]{Example}

\definecolor{turquoise2}{rgb}{0,0.898039,0.933333}
\definecolor{magenta}{rgb}{1,0,1}
\definecolor{olivedrab}{rgb}{0.419608,0.556863,0.137255}
\definecolor{purple2}{rgb}{0.568627,0.172549,0.933333}
\definecolor{amethyst}{rgb}{0.6, 0.4, 0.8}
\definecolor{ao(english)}{rgb}{0.0, 0.5, 0.0}
\definecolor{atomictangerine}{rgb}{1.0, 0.6, 0.4}
\definecolor{amber(sae/ece)}{rgb}{1.0, 0.49, 0.0}
\definecolor{alizarin}{rgb}{0.82, 0.1, 0.26}
\definecolor{auburn}{rgb}{0.43, 0.21, 0.1}
\definecolor{aqua}{rgb}{0.0, 1.0, 1.0}

\makeatletter
\@namedef{subjclassname@2020}{%
  \textup{2020} Mathematics Subject Classification}
\makeatother

\begin{document}
%
%***********************************************************************************************
%***********************************************************************************************
\title{Tensor product of evolution algebras}

\author[Y. Cabrera ]{Yolanda Cabrera Casado}
\address{Departamento de Matem\'atica Aplicada, E.T.S. Ingenier\'\i a Inform\'atica, Universidad de M\'alaga, Campus de Teatinos s/n. 29071 M\'alaga.   Spain. }
\email{yolandacc@uma.es}

\author[D. Mart\'{\i}n]{Dolores Mart\'in Barquero}
\address{Departamento de Matem\'atica Aplicada. Escuela de Ingenier\'\i as Industriales. Universidad de M\'alaga, Campus de Teatinos. 29071 M\'alaga,   Spain.}
\email{dmartin@uma.es}

\author[C. Mart\'{\i}n]{C\'andido Mart\'in Gonz\'alez}
\address{Departamento de \'{A}lgebra, Geometr\'{i}a y Topolog\'{\i}a, Universidad de M\'{a}laga, Spain}
\email{candido\_m@uma.es}

\author[A. Tocino]{Alicia Tocino Sánchez}
\address{Departamento de \'{A}lgebra, Geometr\'{i}a y Topolog\'{\i}a, Universidad de M\'{a}laga, Spain}
\email{alicia.tocino@uma.es}

\thanks{ The  authors are supported by the Spanish Ministerio de Ciencia e Innovaci\'on through project  PID2019-104236GB-I00 and  by the Junta de Andaluc\'{\i}a  through projects  FQM-336 and UMA18-FEDERJA-119,  all of them with FEDER funds. 
}

%***********************************************************************************************
\begin{abstract}
The starting point of this work is that the class of evolution algebras over a fixed field is closed under tensor product.
This arises questions about the inheritance of properties from the tensor product to the factors and conversely. 
For instance nondegeneracy, irreducibility, perfectness and simplicity are investigated. The four-dimensional case is illustrative and useful to contrast conjectures so we achieve a complete classification of four-dimensional perfect evolution algebras arising as tensor product of two-dimensional ones. We find that there are $4$-dimensional evolution algebras which are the tensor product of two nonevolution algebras. 

\end{abstract}

%************************************************************************
\maketitle
%************************************************************************

%************************************************************************
%************************************************************************
\section{Introduction} \label{sec01}
The necessity to address the problem of the genetic inheritance in Biology from a mathematical point of view has made it essential to introduce tools of abstract algebra.
The results of the first studies in genetic from an algebraic point of view of the genetic inheritance behavior, due to I.M.H. Etherington, can be found in  \cite{Et9}, \cite{Et10} and \cite{Et11}. We can locate a background of algebras in genetics in \cite{Wo} and \cite{Reed}. In 2006 arose researches, by J.P. Tian and  P. Vojtechovsky, concerning non-Mendelian genetic inheritance in which it was necessary to introduce a new type of algebra emerging in the study of genetics called evolution algebras (see \cite{Tian} and \cite{tian}). From this last paper, a  wide amount of works have arisen about evolution algebras, see for example \cite{ELduqueGraphs}, \cite{Casas},  \cite{CSV} and \cite{YolandaTesis}. The paper is organized as follows. After a section of preliminary results we study, in the third and fourth sections, the tensor product of algebras and the associated matrices. Perfectness and nondegenerancy are inherited from the tensor product to the factors and conversely (in arbitrary dimension) in Proposition \ref{perfect} and Proposition \ref{rem:zuru}. We give an example of anticommutative algebras (not evolution algebras) whose tensor product is an evolution algebra. However if the ground field is algebraically closed, any four-dimensional tensorially decomposable simple evolution algebra is the tensor product of (simple) evolution algebras (Proposition \ref{simplevo}). It is also related the annihilator of the tensor product with the annihilator of the factors (Lemma \ref{ann}). In section 5 we use graphs theory and the fact to associate a graph to an evolution algebra described in \cite{ELduqueGraphs}.
We study irreducibility in the context of tensor products (Corollary \ref{cor:strong}). In section 6 we prove that when one of the factors is perfect and has an ideal of codimension $1$ then if the tensor product is an evolution algebra, the other factor is an evolution algebra (Proposition \ref{proposition:uf}). We compute the number of zeros $z$ (and of zeros in the diagonal $z_d$) of the Kronecker product in terms of the corresponding numbers $z$ and $z_d$ of the factors. This allows to screen the $4\times 4$ matrices which arise as Kronecker product of $2\times 2$ matrices. In this section we classify four-dimensional tensorially decomposable perfect evolution algebras into $13$ classes and determine some sets of complete invariants, some of them based on characteristic and minimal polynomials (Theorem \ref{clasifica}). Finally we describe an example of classification task for this class of algebras.

%
%\setcounter{theorem}{0}

%************************************************************************
%************************************************************************

\section{Preliminaries} \label{sec02}
%\setcounter{theorem}{0}
%************************************************************************
%************************************************************************

We start recalling several definitions concerning evolution algebras from \cite{YolandaTesis}.
An \textit{evolution algebra} over a field $\K$ is a $\K$-algebra $A$ provided with a basis $B = \{e_i\}_{i \in \Lambda}$ such that $e_i e_j = 0$ whenever $i\neq j$. Such a basis $B$ is called a \textit{natural basis}.
Let $A$ be an evolution algebra, and fix a natural basis $B$ in $A$. The scalars $\omega_{ki} \in\K$ such that
$e_i^2:=e_i e_i=\sum_{k\in\Lambda}\omega_{ki}e_k$ are called the \textit{structure constants} of $A$ relative to $B$, and the matrix $M_B := (\omega_{ki})$ is said to be the \textit{structure matrix} of $A$ relative to $B$. We recall that an algebra $A$ is said to be \textit{perfect} if $A^2=A$, equivalently, the determinant of the structure matrix is not zero. A \textit{reducible} evolution algebra is an evolution algebra $A$ which can be
decomposed as the direct sum of two non-zero
evolution algebras. An evolution algebra which is not reducible will be called \textit{irreducible}.
We say that an algebra $A$ is \textit{simple} if $A^2\neq 0$ and $0$ is the only proper ideal.  
An evolution algebra $A$ is \textit{nondegenerate} if it has a natural basis
$B = \{e_i\}_{i \in\Lambda}$ such that $e_i^2\neq 0$
for every $i\in\Lambda$ (as proved in \cite[Corollary 1.5.4]{YolandaTesis} the definition of nondegenerancy does not depend on the chosen natural basis).
We will say that an algebra $A$ is \textit{anticommutative} if and only if $xy=-yx$ for any $x,y\in A$.\par
We will say that an algebra $A$ is a {\it zero-square algebra} if $x^2=0$ for any $x\in A$. Observe that a zero-square algebra is necessarily anticommutative. If the characteristic of the ground field $\K$ is other than $2$, then the $\K$-algebra $A$ is anticommutative if and only if $A$ is a zero-square algebra.
We denote by $\N^*$ the natural numbers except $0$.\par

Now, we recall the definition of tensor product of two vector spaces.
The \textit{tensor product} of two vector spaces $V$ and $W$ over a field $\K$ is a vector space defined as the set $$V\otimes W=\left\{\displaystyle\sum_{i=1}^{n}\lambda_{i}v_{i}\otimes w_{i} \,\colon \,\lambda_{i} \in \K, \, v_{i}\in V, \, w_{i}\in W,\, n\in\N^*\right\}$$ 

\noindent
where $v\otimes w $ is a bilinear map
\begin{align*}
  v\otimes w  \colon  V^\ast\times W^\ast &\to \K \\
  (f,g) &\mapsto f(v)g(w)
\end{align*}

$\forall v\in V$, $\forall w\in W$ and $V^\ast$ and $W^\ast$ denote the dual vector space of $V$ and $W$ respectively. 
Moreover, if $\{e_{i}\}_{i\in \Lambda}$ is a basis of $V$ and $\{f_{j}\}_{j\in \Gamma}$ is a basis of $W$, then $\{e_{i}\otimes f_{j}\}_{i\in\Lambda, j\in\Gamma}$ is a \textit{basis} of the tensor product $V \otimes W$.

Finally we add here a comment on ideals of $4$-dimensional algebras: If $A$ is a $\K$-algebra of $\dim(A)=4$ and $I,J\triangleleft A$ with $\dim(I)=2$ and $\dim(J)=3$ then either $I\subset J$ or $I+J=A$. Indeed $I\cap J\ne 0$ (otherwise $\dim(I+J)=5$ which is not possible). Thus $\dim(I\cap J)=1,2$ or $3$. In case $\dim(I\cap J)=3$ we have $\dim(I+J)=2$ and since $J\subset I+J$, this is not possible.
If $\dim(I\cap J)=2$ then $\dim(I+J)=3$ whence $J=I+J$ implying
$I\subset J$. Finally if
$\dim(I\cap J)=1$ we have $\dim(I+J)=4$ so that $I+J=A$.

\section{Generalities of tensor product of algebras}

In this section we will consider the tensor product of algebras over fields. As it is well known this product is commutative and associative (up to natural isomorphism). We will define tensorially decomposable algebras and investigate the way in which certain properties can be transferred from the factors to the tensor product and/or viceversa.

\begin{definition}\label{def:product}\rm
 Let $A_1$ and $A_2$ be two $\K-$algebras with bases $\{e_{i}\}_{i\in \Lambda}$ and $\{f_{j}\}_{j\in \Gamma}$ respectively. We define a product on $A_1\otimes A_2$ as follows:
 $$(e_{i}\otimes f_{j})\cdot(e_{k}\otimes f_{r})= e_{i}e_{k}\otimes f_{j}f_{r}.$$
\end{definition}

\begin{remark}\rm  If $A_1$ and $A_2$ are both commutative algebras or anticommutative algebras, then $A_1\otimes A_2$ is commutative. Since
zero-square algebras are anticommutative we have: if both algebras $A_1$ and $A_2$ are commutative or both zero-square algebras, then their tensor product is a commutative algebra.
\end{remark}

\begin{definition}\rm
 We say that a $\K$-algebra $A$ is \textit{tensorially decomposable} if it is isomorphic to $A_1\otimes A_2$ where $A_1$ and $A_2$ are $\K$-algebras with $\dim(A_1),\dim(A_2)> 1$. Otherwise we say that $A$ is \textit{tensorially indecomposable}.
\end{definition}

We use the terms tensorially decomposable and tensorially indecomposable also for matrices when they come or not from the Kronecker product of two matrices.

\begin{proposition}\label{laqeme}
Let $A_1$ and $A_2$ be $\K$-algebras such that $A_1\otimes A_2$ is commutative and $(A_1\otimes A_2)^2\ne 0$, then either $A_1$ and $A_2$ are commutative or
both are zero-square algebras.
\end{proposition}
\begin{proof}
Assume that $A_1$ or $A_2$ is not a zero-square algebra. Without lost of generality we may
assume that $A_1$ is not zero-square and take $a\in A_1$ such
that $a^2\ne 0$. Then for any $x,y\in A_2$ we have 
$a^2\otimes xy=(a\otimes x)(a\otimes y)=(a\otimes y)(a\otimes x)=a^2\otimes yx$ whence $a^2\otimes (xy-yx)=0$ which implies $xy=yx$. So $A_2$ is commutative.
We know that $A_2^2\ne 0$ because $(A_1\otimes A_2)^2\ne 0$.
Take $b_1,b_2\in A_2$ such that $b_1b_2\ne 0$. Then for any
$a_1,a_2\in A_1$ we have 
$$a_1a_2\otimes b_1b_2=(a_1\otimes b_1)(a_2\otimes b_2)=
(a_2\otimes b_2)(a_1\otimes b_1)=a_2a_1\otimes b_2b_1$$
hence by commutativity of $A_2$, $a_1a_2=a_2a_1$ so that $A_1$ is also commutative. 
\end{proof}

It is possible to define a commutative $\K$-algebra structure by means of a collection of inner products defined on the underlying $\K$-vector space following \cite{CCMM}. Let $A$ be a commutative finite-dimensional $\K$-algebra with basis $\{e_i\}$. We consider the inner product  $\esc{\cdot,\cdot}_i\colon A \times A \to \K$ such that $xy=\sum_i\esc{x,y}_ie_i$ for any $x,y\in A$. Fixing the basis, these inner products determine precisely the algebra product and so they will be called \textit{structure inner products}.

\begin{lemma}
Let $A_1$ and $A_2$ be two commutative finite-dimension algebras with basis $\{e_i\}_{i=1}^n$ and $\{f_j\}_{j=1}^m$ respectively. We consider the structure inner products of $A_1$ and $A_2$ in the given basis. So, we have $x_1y_1=\sum_i \esc{ x_1, y_1}_i e_i  $ and $x_2y_2=\sum_j \esc{ x_2, y_2}'_j f_j  $ for every $x_1,y_1 \in A_1$ and $x_2,y_2 \in A_2$. If we take into account the product defined by Definition \ref{def:product}, let $M_s$ be the matrix of structure inner product $\esc{\cdot,\cdot}_s'$ relative to the basis $\{f_j\}_{j=1}^m$ for every $s$. Then the matrix of structure inner product $\esc{\cdot,\cdot}^{\otimes}_{r,s}$ relative to the basis $\{e_i \otimes f_j\}$ is of following one (except reordering):
$$M_{\esc{\cdot,\cdot}_{r,s}^\otimes}=\begin{pmatrix}
\esc{e_1,e_1}_r M_s &  \esc{e_1,e_2}_r M_s &\cdots & \esc{e_1,e_n}_r M_s \\
\esc{e_2,e_1}_r M_s &\esc{e_2,e_2}_r M_s &\cdots & \esc{e_2,e_n}_r M_s \\
\vdots  &\cdots& \\
\esc{e_n,e_1}_r M_s & \esc{e_n,e_2}_r M_s 
& \cdots & \esc{e_n,e_n}_r M_s 
\end{pmatrix}.$$
\end{lemma}

\begin{proof}
If we reorder the basis $\{e_i \otimes f_j\}$ in the way $\{e_1\otimes f_1,\ldots e_1 \otimes f_m,\ldots, e_n\otimes f_1, \ldots,e_n \otimes f_m\}$ then $(e_i\otimes f_j)^2=e_i^2 \otimes f_j ^2 = \sum_k \sum _l \esc{e_i,e_i}_k \esc{f_j,f_j}'_l e_k \otimes f_l$. Then the projection over $e_r \otimes e_s$ is $\esc{e_i,e_i}_r \esc{f_j,f_j}'_s$. So, we obtain that the matrix of structure inner product $\esc{\cdot,\cdot}_s'$ relative to the basis $\{f_j\}_{j=1}^m$ is the desired one.
\end{proof}

Note that the matrices of the structure inner products of $A_1 \otimes A_2$ are obtained as the Kronecker product of each structure inner product matrix of $A_1$ by each  structure inner product matrix of $A_2$.

\begin{remark} \label{propi}\rm
Recall that for $V$ and $W$  vector spaces and  $\{w_i\,|\,i\in I\}$ a linearly independent set of $W$ then $\sum_i v_i\otimes w_i=0$ with $v_i\in V$ for all $i\in I$, implies $v_i=0$ for all $i\in I$. 
\end{remark}

\begin{lemma} \label{remark:tensorequal}

Let $A_1$ and $A_2$ be two $\K$-vector spaces of arbitrary dimension. If $U_1, V_1$ are subspaces of $A_1$ and  $U_2, V_2$ are subspaces of $A_2$ with $U_i \subset V_i$ ($i=1,2$), then $U_1\otimes U_2 = V_1 \otimes V_2$ if and only if $U_1=V_1$ and $U_2=V_2$. 
\end{lemma}
\begin{proof}
Consider $U_1=\spa\{a_i\,|\,i\in I\}$, $V_1=\spa(\{a_i\,|\, i\in I\} \dot{\bigcup}\{a_j\,|\,j\in J\})$, $A_1=\spa(\{a_i\,|\, i\in I\}\dot{\bigcup} \{a_j\,|\,j\in J\}\dot{\bigcup}\{a_h\,|\, h\in H\})$ and
$U_2=\spa\{b_p\,|\,p\in P\}$,
$V_2=\spa(\{b_p\,|\, p\in P\} \dot{\bigcup}\{b_q\,|\,q\in Q\})$, $A_2=\spa(\{b_p\,|\, p\in P\}\dot{\bigcup} \{b_q\,|\,q\in Q\}\dot{\bigcup}\{b_r\,|\, r\in R\})$.
Now we take $x\in V_1$ and $b_k\in V_2$ with $k\in P$ and consider $x\otimes b_k\in V_1\otimes V_2=U_1\otimes U_2$. So $x\otimes b_k=\sum_{i\in I, p\in P}\lambda_{ip}^{(k)}a_i\otimes b_p=\sum_{p\in P}\sum_{i\in I}\lambda_{ip}^{(k)}a_i\otimes b_{p}$.
If we write $z_{kp}=\sum_{i\in I}\lambda_{ip}^{(k)}a_i$ then 
$x\otimes b_{k}=\sum_{p\in P}z_{kp}\otimes b_p$. So $(x-z_{kk})\otimes b_k-\sum_{p\in P, p\neq k} z_{kp}\otimes b_p=0$.

As $\{b_p\,|\,p\in P\}$ is a basis of $U_2$ for Remark \ref{propi}  then $x=z_{kk}\in U_1$. Analogously it can be proved that $V_2\subset U_2$.\end{proof}

For finite dimensional vector spaces the previous lemma can be proved by using a dimensional argument.

\begin{proposition} \label{perfect}
Let $A_1$ and $A_2$ be two algebras. We consider the tensor product $A_1 \otimes A_2$. Then $A_1$ and $A_2$ are perfect  algebras if and only if $A_1 \otimes A_2$ is a perfect algebra.
\end{proposition}

\begin{proof}
Let $B_1=\{e_i\,|\ i\in I\}$ be a  basis of $A_1$ and $B_2=\{f_j\,|\,j\in J\}$ a  basis of $A_2$. We consider $B_1\otimes B_2=\{e_i\otimes f_j\,|\,i\in I, j\in J\}$ a  basis of $A_1\otimes A_2$. 
If we suppose that $A_1$ and $A_2$ are perfect, then $A_1=A_1^2$ and $A_2=A_2^2$.
Moreover $e_i=\sum_q \lambda_{qi}e_q^2$ and $f_j=\sum_t\mu_{tj} f_t^2$.
To see that $A_1\otimes A_2$ is perfect we need to show that $A_1\otimes A_2\subset (A_1\otimes A_2)^2$. 
Indeed, as
$e_i\otimes f_j=\sum_{q,t} \lambda_{qi} \, \mu_{tj}e_q^2\otimes f_t^2=\sum_{q,t} \lambda_{qi} \, \mu_{tj}(e_q\otimes f_t)^2\in (A_1\otimes A_2)^2$ then we have the statement.
 
For the converse, we first prove that 
$(A_1\otimes A_2)^2=A_1^2\otimes A_2^2$. For instance, to prove
the inclusion $(A_1\otimes A_2)^2\subset A_1^2\otimes A_2^2$ we take
$z\in (A_1\otimes A_2)^2$. 
Then $z=\sum_{ijkl}\lambda_{ijkl}(e_i\otimes f_j)(e_k\otimes f_l)=\sum_{ijkl}\lambda_{ijkl}e_i e_k\otimes f_j f_l\in A_1^2\otimes A_2^2$.

For the other inclusion, since $A_1^2\subset A_1$ and $A_2^2\subset A_2$ then $A_1^2\otimes A_2^2\subset A_1\otimes A_2$.
Consequently, as $A_1\otimes A_2$ is perfect we
have $A_1^2\otimes A_2^2 \subset A_1\otimes A_2=(A_1\otimes A_2)^2$. 

So far we know $(A_1\otimes A_2)^2=A_1^2\otimes A_2^2=A_1 \otimes A_2$, by Lemma \ref{remark:tensorequal} we get $A_1^2=A_1$ and $A_2^2=A_2$ as we wanted.
\end{proof}

\begin{remark}
\rm
Not every ideal of the tensor product of two evolution algebras is a tensor product of ideals. Consider $A$ an arbitrary evolution algebra and $A'$ an evolution algebra with $A'^2=0$. Let $(S,+)$ be a vector subspace of $(A,+)$, not an ideal $S\ntriangleleft A$. One can conclude that $S\otimes A' \triangleleft A\otimes A'$. Using the previous remark one can conclude that if $S\otimes A'=I\otimes I'$ with $I\triangleleft A$ and $I'\triangleleft A'$ we get that $S=I$ and $A'=I'$.
\end{remark}

\begin{remark}\label{remark:simple}\rm 
If $A_1$ and $A_2$ are $\K$-algebras and $A_1\otimes A_2$ is simple, then both $A_1$ and $A_2$ are simple algebras: if $0\ne I$ is a proper ideal of $A_1$, then  $0\ne I\otimes A_2\triangleleft A_1\otimes A_2$ is also proper by Lemma  \ref{remark:tensorequal}.
However, there are simple algebras (even evolution algebras) whose tensor product is not simple. For instance, take $\K$ to be a field of
characteristic other than $2$ and $A$ to be the $2$-dimensional $\K$-algebra with a basis $\{e_1,e_2\}$ such that 
$e_2e_1=e_1e_2=0$, $e_1^2=e_2$ and $e_2^2=e_1$. Observe that $A\otimes A$ is not simple because it is a perfect evolution algebra by Proposition \ref{perfect} and its associated graph is not strongly connected (see \cite[Proposition 2.7]{CKS}). 
\end{remark}

\section{The tensor product of evolution algebras}

Now, we focus on the particular case of tensor product of evolution algebras which is again an evolution algebra. 
Indeed, we classify the $4-$dimensional evolution algebras that are tensorially decomposable. When $\K$ is algebraically closed, and $A$ is simple, we investigate which are its factors. We also study the annihilator of the tensor product of evolution algebras. 

\begin{proposition}
If $A_1$ and $A_2$ are evolution $\K$-algebras, then $A_1\otimes A_2$ is also an evolution $\K$-algebra. Furthermore, if $B_1=\{e_{i}\}_{i \in \Lambda}$ be a natural basis of $A_1$ and let $B_2=\{f_{j}\}_{j \in \Gamma}$ be a natural basis of $A_2$, then $\{e_{i}\otimes f_{j}\}_{i\in\Lambda, j\in\Gamma}$ is a natural basis of $A_1 \otimes A_2$.
\end{proposition}

\begin{proof}
Let $B_1=\{e_{i}\}_{i \in \Lambda}$ be a natural basis of $A_1$ and let $B_2=\{f_{j}\}_{j \in \Gamma}$ be a natural basis of $A_2$. We are going to see that the basis $\{e_{i}\otimes f_{j}\}_{i\in\Lambda, j\in\Gamma}$ is a natural basis of $A_1 \otimes A_2$ i.e., $(e_i\otimes f_j)(e_k \otimes f_s)=e_i e_k \otimes f_j f_s=0$ for $i \neq k$ or $j\neq s$. Let $(h,g) \in A_1^\ast \times A_2^\ast$ then $(e_i\otimes f_j)(e_k \otimes f_s)(h,g)=h(e_i e_k)g(f_j f_s)=0$ if $i \neq k$ or $j\neq s$ since $B_1$ and $B_2$ are natural bases and $h$ and $g$ are linear maps. Therefore, as $(h,g)$ is an arbitrary element of $A_1^\ast \times A_2^\ast$, we get that $e_i e_k \otimes f_j f_s \equiv 0$ for  $i \neq k$ or $j\neq s$.
\end{proof}

\begin{remark}\rm 
({\bf Structure matrix of tensor product of finite dimensional evolution algebras}) 
Suppose $A_1$ and $A_2$ are two finite dimensional evolution $\K$-algebras with natural basis $B_1=\{e_1,\ldots, e_n\}$ and $B_2=\{f_1,\ldots,f_m\}$ respectively. Let $M_{B_1}=(\omega_{ij})$ and $M_{B_2}=(\beta_{lm})$ be the structure matrices associated to $A_1$ and $A_2$ respectively. Then, the structure matrix of the evolution algebra $A_1\otimes A_2$ relative to the basis $B_1\otimes B_2=\{e_1\otimes f_1,\ldots,e_1 \otimes f_m,\ldots,e_n\otimes f_1, \ldots, e_n \otimes f_m \}$ is the Kronecker product of $M_{B_1}$ and $M_{B_2}$, i.e., $M_{B_1\otimes B_2}=M_{B_1}\otimes M_{B_2}$. Indeed, as
$(e_i\otimes f_j)^2= e_i^2 \otimes f_j^2= \sum_{k=1}^n \omega_{ki} \sum_{t=1}^m \beta_{tj} e_k \otimes  f_t$ then we obtain that $M_{B_1\otimes B_2}=M_{B_1}\otimes M_{B_2}$.
\end{remark}

\begin{lemma}\label{p2}
  Let $A_1$ and $A_2$ be evolution algebras, $B_1$ and $B'_1$ natural basis of $A_1$ and $B_2$ and $B'_2$  natural basis of $A_2$. We denote by $P$ the basis change matrix of $A_1$ from ${B_1}$ to ${B'_1}$ and $Q$ the basis change matrix of $A_2$ from ${B_2}$ to ${B'_2}$. Hence, $$M_{B_1\otimes B_2}=(P\otimes Q)^{-1}(M_{B'_1}\otimes M_{B'_2})(P\otimes Q)^{(2)}.$$
\end{lemma}

\begin{proof}
The proof is a consequence of the relation between two structure matrices of an evolution algebra and the facts that if $P$ and $Q$ be square matrices, then $(P\otimes Q)^{(2)}=P^{(2)}\otimes Q^{(2)} $ and $(P\otimes Q)^{-1}= P^{-1}\otimes Q^{-1}$. 
\end{proof}

\begin{remark}\label{nonperfect}\rm
Recall that if $M_1 \in \mathcal{M}_n(\K)$ and $M_2\in \mathcal{M}_m(\K)$, then $ \vert M_1 \otimes M_2 \vert = \vert M_1 \vert^m \vert M_2\vert^n$ (see \cite[Corollary p.6]{SYLVESTER}).
So,  $M_1\otimes M_2$ is a non perfect evolution algebra if and only if either $M_1$ or $M_2$ is also a non perfect evolution algebra.  
\end{remark}

\begin{example} \rm
Let us build the tensor product of two  evolution algebras. Consider $A_1$ an evolution algebra with natural basis $B_1=\{e_1,e_2,e_3\}$ and structure matrix
$M_{B_1}=\tiny\begin{pmatrix}
 1 & -1 & 0\\
  1 & -1 & -1  \\
0 & 0 & 1  
\end{pmatrix}.$
Consider $A_2$ another evolution algebra with natural basis $B_2=\{f_1,f_2,f_3\}$ and structure matrix
$M_{B_2}=\tiny\begin{pmatrix}
 0 & 1 & 0\\
 1 & 0 & 0\\
 0 & 0 & 1
\end{pmatrix}.$
Hence, the tensor product $A_1\otimes A_2$ is the evolution algebra with basis $B_1\otimes B_2=\{e_1\otimes f_1, e_1\otimes f_2, e_1\otimes f_3, e_2\otimes f_1, e_2\otimes f_2, e_2\otimes f_3, e_3\otimes f_1, e_3\otimes f_2,e_3\otimes f_3\}$ and structure matrix\\
$$M_{B_1\otimes B_2}=\begin{pmatrix}
 1 M_{B_2} & -1 M_{B_2} & 0 M_{B_2}\\
 1 M_{B_2} & -1 M_{B_2} & -1 M_{B_2}\\
 0 M_{B_2} & 0 M_{B_2} & 1 M_{B_2}
 \end{pmatrix}=\tiny
\left(
\begin{array}{ccc|ccc|ccc}
 0 & 1 & 0 & 0 & -1 & 0 & 0 & 0 & 0\\
 1 & 0 & 0 & -1 & 0 & 0 & 0 & 0 & 0 \\
 0 & 0 & 1 & 0 & 0 & -1 & 0 & 0 & 0 \\ \hline
 0 & 1 & 0 & 0 & -1 & 0 & 0 & -1 & 0\\
 1 & 0 & 0 & -1 & 0 & 0 & -1 & 0 & 0\\
 0 & 0 & 1 & 0 & 0 & -1 & 0 & 0 & -1\\ \hline
 0 & 0 & 0 & 0 & 0 & 0 & 0 & 1 & 0 \\
 0 & 0 & 0 & 0 & 0 & 0 & 1 & 0 & 0 \\
 0 & 0 & 0 & 0 & 0 & 0 & 0 & 0 & 1\\
 \end{array}
\right).$$
\end{example}

\begin{example}\label{dossiete}\rm 
Let $\K$ a field with characteristic different from two (we denote by $\car(\K)$ the characteristic of $\K$) and consider the $2$-dimensional $\K$-algebra $A_1$ with basis $\{e_1,e_2\}$ such that $e_1^2=e_2^2=0$ and $e_1e_2=e_2=-e_2e_1$. 
It can be checked that this is not an evolution algebra by computing the evolution test ideal $I$ and seeing that $1\in I$ (see \cite{CCMM}). Indeed the evolution test ideal is $I=(xt-yz,w(xt-yz)-1)\triangleleft \K[x,y,z,t,w]$ and clearly $1\in I$. Hence $A_2$ is not an evolution algebra. However $A_2:=A_1\otimes A_1$ is an evolution algebra as can be proved by applying the evolution test ideal in the papers \cite{CCMM} and \cite{CCMM2}. First we consider $u_1=e_1\otimes e_1$, $u_2=e_1\otimes e_2$, $u_3=e_2\otimes e_1$ and $u_4=e_2\otimes e_2$. The multiplication table of $A_2$ relative to the
basis $\{u_i\}_1^4$ is 
\begin{center}
\begin{tabular}{|c|c|c|c|c|}
\hline 
$\cdot$ & $u_1$ & $u_2$ & $u_3$ & $u_4$\\
\hline 
$u_1$ & $0$ & $0$ & $0$ & $u_4$\\ 
$u_2$ & $0$ & $0$ & $-u_4$ & 0\\
$u_3$ & $0$ & $-u_4$ & $0$ & $0$\\
$u_4$ & $u_4$ & $0$ & $0$ & $0$\\
\hline
\end{tabular}
\end{center}
and then a natural basis for $A_2$ is $\{u_1+u_4,u_1-u_4,u_2+u_3,u_2-u_3\}$. Hence there is a four-dimensional evolution algebra which is the tensor product of $2$-dimensional algebras which are not evolution algebras. 
Under the assumption $\car(\K)\ne 2$, it seems that $A_2$ is the unique $4$-dimensional algebra which splits as a tensor product
of $2$-dimensional anticommutative algebras (using Kaygorodov and Volkov classification for algebraically closed fields in \cite{Ivan} 
the unique $2$-dimensional algebra in which any element has zero square is $A_1$ up to isomorphism).
Observe that $A_2$ is not simple hence it is worth to investigate if there exist $4$-dimensional simple evolution algebras which split as the tensor product of $2$-dimensional algebras which are not evolution algebras. 
\end{example}

Now we assume that $A=A_1\otimes A_2$ is a simple evolution algebra with $\dim(A_i)=2$ for $i=1,2$. By Proposition \ref{laqeme} we know that both $A_1$ and $A_2$ are commutative or both are zero-square algebras (and simple because their tensor product is simple, Remark \ref{remark:simple}). We assume for now that $\K$ is algebraically closed. So we can determine those algebras from the classification in \cite[Table 1]{Ivan} which are commutative: 

\newpage
\begin{center}\footnotesize
\begin{tabular}{|c|c|c|c|c|c|c|c|}
\hline 
$\begin{matrix} \A_1(\frac{1}{2})\\ \text{Not simple}
\end{matrix}$ & $\begin{matrix} \A_2\ \text{if}\ $\car=2$\\\text{Not simple}\end{matrix}$ & $\begin{matrix} \A_3\\\text{Not simple}\end{matrix}$ & 
$\begin{matrix} \A_4(0)\ \text{if}\ \car=2\\\text{Simple}\end{matrix}$ & $\begin{matrix} \mathsf{B}_2(\frac{1}{2})\\\text{Not simple}\end{matrix}$ & $\begin{matrix}
\mathsf{B}_3\ \text{if}\ $\car=2$\\\text{Not simple} \end{matrix}$ & $\begin{matrix}
\mathsf{C}(\frac{1}{2},0)\ \\\text{Simple} \end{matrix}$ &  $\begin{matrix} \mathsf{D}_1(\frac{1}{2},0)\\ \text{Not simple} \end{matrix}$\\
\hline
\end{tabular} 
\end{center}

\begin{center}\footnotesize
\begin{tabular}{|c|c|c|c|}
\hline 
$\begin{matrix} \mathsf{D}_2(\alpha,\alpha), \alpha\neq \frac{1}{2}\\ \text{Not simple} \end{matrix}$ & $\begin{matrix} \mathsf{D}_3(\alpha,\alpha)\text{ if}\ $\car=2$\\ \text{Simple if}\ \alpha\ne 0
\end{matrix}$ & $\begin{matrix} \mathsf{E}_1(\alpha,\beta,\alpha,\beta), 4\alpha\beta\neq 1, 2\beta\neq 1, 2\alpha\neq 1
\\\text{Simple if } \alpha+\beta\ne 1\\ \alpha,\beta\ne 0\end{matrix}$ & $\begin{matrix} \mathsf{E}_2(\frac{1}{2},\beta,\beta), \\\text{Simple if }\beta\ne\frac{1}{2},0 \end{matrix}$ \\
    \hline 
    \end{tabular} 
\end{center}

\begin{center}\footnotesize
\begin{tabular}{|c|c|}
\hline 
 $\begin{matrix} \mathsf{E}_3(\frac{1}{2},\frac{1}{2},\gamma), \gamma\ne 0\\\text{Simple if }\gamma\ne 1\end{matrix}$ &  $\begin{matrix} \mathsf{E}_5(\frac{1}{2})\\ \text{Not simple}\end{matrix}$  \\
    \hline 
    \end{tabular} 
\end{center}

The only two-dimensional algebras that are zero-square is $B_3$.
The unique commutative simple nonevolution algebras of the above table are: $\A_4(0)$, $\mathsf{D}_3(\alpha,\alpha)$ and $\mathsf{E}_3(\frac{1}{2},\frac{1}{2},\gamma)$. 
To check if an algebra is an evolution algebra we use the evolution test ideal (see \cite{CCMM}) computing Groebner basis of that ideal.
We summarize all the information in the following table.

\begin{center}\tiny
\begin{tabular}{|c|c|c|}
\hline
Simple algebra & Evolution? & Nat. Basis $\{e_1,e_2\}$\\
\hline 
$\A_4(0)$, $\car=2$ & FALSE & \\
\hline
$C(\frac{1}{2},0)$  & TRUE & $\small\begin{cases}e_1+e_2\\e_1-e_2\end{cases}$\\
\hline
$\begin{matrix}\mathsf{D}_3(\alpha,\alpha), \car=2\cr \alpha\ne 0\end{matrix}$ & FALSE  &\\
\hline
$\begin{matrix}\mathsf{E}_1(\alpha,\beta,\alpha,\beta)\cr \alpha,\beta\ne 0,\frac{1}{2},\alpha\beta\ne\frac{1}{4}\\
\alpha+\beta\ne 1\end{matrix}$ & TRUE & $\small\begin{cases}
e_1+\frac{\Delta-1}{2\alpha}\ e_2\\
-\frac{\alpha}{\Delta}e_1+\frac{1}{2}(1+\frac{1}{\Delta})\ e_2\\\Delta=\sqrt{1-4\alpha\beta}
\end{cases}$
\\
\hline
$\begin{matrix}\mathsf{E}_2(\frac{1}{2},\beta,\beta)\cr \beta\ne 0,\frac{1}{2}\end{matrix}$ & TRUE & $\small\begin{cases}
e_1+(\sqrt{1-2\beta}-1)\ e_2\\
e_1-(\sqrt{1-2\beta}+1)\ e_2
\end{cases}$\\
\hline
$\begin{matrix}\mathsf{E}_3(\frac{1}{2},\frac{1}{2},\gamma)\cr \gamma\ne 0,1\end{matrix}$ & FALSE  &\\
\hline
\end{tabular}
\end{center}

Next, we investigate if there is or not $4$-dimensional simple evolution algebras $A$ which split as the tensor products of two necessarily simple non evolution algebras using the simultaneous diagonalization algorithm explained in \cite{CCMM}.
Thus in the decomposition
$A=A_1\otimes A_2$ we have both $A_1,A_2$ isomorphic either to $\A_4(0)$, $\mathsf{D}_3(\alpha,\alpha)$ or to $\mathsf{E}_3(\frac{1}{2},\frac{1}{2},\gamma)$. One can check that none of the algebras 
$\A_4(0)\otimes \A_4(0)$, $\A_4(0)\otimes \mathsf{D}_3(\alpha,\alpha)$, $\mathsf{D}_3(\alpha,\alpha)\otimes \mathsf{D}_3(\alpha,\alpha)$ and $\mathsf{E}_3(\frac{1}{2},\frac{1}{2},\gamma)\otimes \mathsf{E}_3(\frac{1}{2},\frac{1}{2},\gamma)$ is an evolution algebra.

Moreover, we have check that $\mathsf{E}_3(\frac{1}{2},\frac{1}{2},\gamma)\otimes \mathsf{C}(\frac{1}{2},0)$, $\mathsf{E}_3(\frac{1}{2},\frac{1}{2},\gamma)\otimes \mathsf{E}_1(\alpha,\beta,\alpha,\beta)$ and $\mathsf{E}_3(\frac{1}{2},\frac{1}{2},\gamma)\otimes \mathsf{E}_2(\frac{1}{2},\beta,\beta)$ are not evolution algebras.

\begin{proposition}\label{simplevo}
If $\K$ is algebraically closed and $A$ is a four-dimensional simple evolution algebra with $A=A_1\otimes A_2$, then both factors $A_1$ and $A_2$ are evolution algebras.
\end{proposition}
\begin{proof}
If some of $A_1$ or $A_2$ is one-dimensional, then the decomposition is trivial (the ground field, which is an evolution algebra, tensor product with something isomorphic to $A$, hence an evolution algebra). Next we assume $\dim(A_i)=2$ for $i=1,2$.
By Proposition \ref{laqeme} both factors are commutative or zero square algebras.
In the commutative case, the discussion is made above. 
If both factors are zero square algebras, the only possibility is the algebra $\mathsf{B}_3$ in \cite[Table 1]{Ivan} but it is not simple.
\end{proof}

\begin{proposition}\label{rem:zuru}
Let $A_1$ and $A_2$ be two evolution algebras. Then $A_1$ and $A_2$ are nondegenerate  if and only if $A_1\otimes A_2$ is nondegenerate. 
\end{proposition}
\begin{proof}
In order to see this, take natural bases $B_1=\{e_i\,|\,i\in \Lambda_1\}$ of
$A_1$ and $B_2=\{f_j\,|\,j\in \Lambda_2\}$ of $A_2$. Assume that $e_i^2=\sum_s\omega_{si} e_s$
and $f_j^2=\sum_t\sigma_{tj} e_t$, so the structure constants of $A_1$ are the
$\omega_{si}$'s (and those of $A_2$ are the $\sigma_{tj}$'s).
We know that $B_1\otimes B_2=\{e_i\otimes f_j\,|\,i\in \Lambda_1,j\in \Lambda_2\}$ is
a natural basis of $A_1\otimes A_2$. If there is some $(i,j)\in \Lambda_1\times \Lambda_2$ such
that $(e_i\otimes f_j)^2=0$, then 
$0=\sum_{s,t}\omega_{si}\sigma_{tj} e_s\otimes f_t$ and we have $\omega_{si}\sigma_{tj}=0$ for any $s$ and $t$. Thus if there is $s\in \Lambda_1$ 
such that $\omega_{si}\ne 0$, then $\sigma_{tj}=0$ for any $t$ contradicting that
$A_2$ is nondegenerate. So $\omega_{si}=0$ for any $s\in \Lambda_1$ contradicting that
$A_1$ is nondegenerate. Reciprocally, suppose that $A_1$ is degenerate, then there exists $k \in \Lambda_1$ such that $e_k^2=0$. Then $(e_k \otimes f_t)^2=0$ for every $t \in \Lambda_2$, so $A_1 \otimes A_2$ is degenerate.
\end{proof}

\begin{lemma}\label{ann}
If $A_i$ are evolution algebras over $\K$ ($i=1,2$), then
$$\ann(A_1\otimes A_2)=A_1\otimes\ann(A_2)+\ann(A_1)\otimes A_2.$$ 
\end{lemma}
\begin{proof}
The nontrivial inclusion is as follows: take natural basis $\{e_i\}_{i\in I}$ of $A_1$ and $\{f_j\}_{j\in J}$ of $A_2$. We assume $e_i^2=\sum_n\omega_{ni} e_n$ and similarly $f_j^2=\sum_m\sigma_{mj} f_m$.
Choose
an arbitrary $z\in\ann(A_1\otimes A_2)$. So $z=\sum_{ij}\lambda_{ij}e_i\otimes f_j$ for some scalars  $\lambda_{ij}\in \K$. But $z(e_p\otimes f_q)=0$ for any $(p,q)\in I\times J$. This implies that $$\lambda_{pq}e_p^2\otimes f_q^2=
\sum_{s,t}\lambda_{pq}\omega_{sp} e_s\otimes \sigma_{tq} f_t=\sum_{s,t}\lambda_{pq}\omega_{sp}\sigma_{tq}(e_s\otimes f_t)=0$$ for any $p\in I$ and $q\in J$. 
Thus $\lambda_{pq}\omega_{sp}\sigma_{tq}=0$ for any $p,s\in I$ and $q,t\in J$.
Next we prove that for any $p\in I$ and $q\in J$ such that $\lambda_{pq}\ne 0$ we have $e_p\in\ann(A_1)$
or $f_q\in\ann(A_2)$: indeed if $e_p\notin\ann(A_1)$, then
some $\omega_{sp}\ne 0$ and then $\lambda_{pq}\sigma_{tq}=0$ for any $q\in J$. Since $\lambda_{pq}\ne 0$ we conclude that $\sigma_{tq}=0$ for any $t$. Thus $f_q^2=0$ and $f_q\in\ann(A_2)$. In a similar way it is proved that if
$f_q\notin\ann(A_2)$, then $e_p\in\ann(A_1)$. Thus
$z\in\ann(A_1)\otimes A_2+A_1\otimes\ann(A_2)$.
\end{proof}

\section{Graph associated to a tensor product of evolution algebras}

We recall some definitions relative to graphs. A \textit{directed graph} is a $4$-tuple $G=(E_G^0, E_G^1, s_G, r_G)$ with $E_G^0$, $E_G^1$  sets and $s_G, r_G: E_G^1 \to E_G^0$ maps. The \textit{vertices} of $G$ are the elements of $E_G^0$ and the \textit{arrows} or \textit{directed edges} of $G$ are the elements of $E_G^1$. For
$f\in E_G^1$ the vertices $r(f)$ and $s(f)$ are called the \textit{range} and the \textit{source} of $f$, respectively. A $v\in E_G^{0}$ is called \textit{sink} if it verifying that $s(f) \neq v$, for every $f\in E_G^1$.  
A finite sequence of arrows $\mu=f_1\dots f_n$  in  $G$, such that $r(f_i)=s(f_{i+1})$ for $i\in\{1,\dots,(n-1)\}$ is called a \textit{path} or a \textit{path from $s(f_1)$ to $r(f_n)$}.  In this case we say that $n$ is the \textit{length} of the path $\mu$ and denote by $\mu^0$ the set $\mu^0:=\{s(f_1),r(f_1),\dots,r(f_n)\}$. Let $\mu = f_1 \dots f_n$ be a path in $G$. If  $ \vert\mu\vert =n\geq 1$, and if $v=s(\mu)=r(\mu)$, then $\mu$ is called a \textit{closed path based at $v$}.  If $\mu = f_1 \dots f_n$ is a closed path based at $v$ and $s(f_i)\neq s(f_j)$ for every $i\neq j$, then $\mu$ is called a \textit{cycle based at} $v$ or simply a \textit{cycle}. A cycle of length $1$ will be said to be a \textit{loop}. Given a finite graph $G$, its \textit{adjacency matrix} is the matrix $M_{G}=(a_{ij})$ where $a_{ij}$ is the number of arrows from $i$ to $j$.

One can associated a graph to an evolution algebra. Let $A$ be an evolution algebra with natural basis $B=\{e_i \}_{i\in \Lambda}$ and structure matrix $M_B=(\omega_{ij})\in  {M}_\Lambda(\mathbb K)$. We consider the matrix
$M=(a_{ij})\in  {M}_\Lambda(\mathbb K)$ such that $a_{ij}=0$ if $\omega_{ij}=0$ and $a_{ij}=1$ if $\omega_{ij}\neq 0$. The directed graph whose adjacency matrix is given by $M=(a_{ij})$ is called the \textit{directed graph associated to the evolution algebra} $A$ (relative to the basis $B$). We denote it by $G_{B}$ (or simply by $G$ if the basis $B$ is understood) and its adjacency matrix by $M_{G_B}$. 
In this way, we consider directed graphs with at most one arrow between two vertices of $E_G^0$. In the literature this is known as a directed graph. A graph is \textit{strongly connected} if given two different vertices there exists a path that goes from the first to the second.

We recall some results about connectivity of directed graphs. 
Consider $\mathcal{B}=(\{0,1\},\lor,\land)$ the standard Boolean algebra ($\lor$ denotes logical OR, and $\land$ denotes logical AND). We denote by $\mathcal{B}^{(n\times n)}$ the set of all $n\times n$ Boolean matrices. Consider $\odot$ the Boolean matrix multiplication, that is, if $A,B\in \mathcal{B}^{(n\times n)}$ then $C=A\odot B\in \mathcal{B}^{(n\times n)}$ with $c_{ij}=\bigvee_{j=1}^{n}a_{ik}\land b_{kj}$. Let $A^0=I$ where $I$ denotes the identity matrix and let $A^i=A\odot A^{i-1}$ for $i >0$ integer. We define $\sigma_k(A)=A\lor A^2 \lor \ldots \lor A^k$ for any $A\in \B$. 
As the number of elements of $\B$ is finite, we have that $\sigma_j(A)=\sigma_i(A)$ for some $j>i$ and $j\leq n$. It is easy to check that if $\s_{k-1}(A)=\s_k(A)$ for some $k \in \N$, then $\s_k(A)=\s_{k+1}(A)$.

\begin{definition} \rm
 For a matrix $A\in \B$ the minimum $k$ such that $\s_k(A)=\s_{k+1}(A)$ will be called the {\it stabilizing index} of $A$. 
\end{definition}

\begin{remark}\label{remark:polinomio}\rm
We recall that if $A$ is a finite dimensional perfect evolution algebra and $B_1$ and $B_2$ are two natural bases of $A$. By \cite[Corollary 4.5]{ELduqueGraphs}) the graphs $G_{B_1}$ and $G_{B_2}$ are isomorphic. So, their adjacency matrices differ by a permutation of rows and columns. Then there exists a permutation matrix $Q$ such that $M_{G_{B_2}}=Q M_{G_{B_1}}Q^{-1}$, that is to say $M_{G_{B_1}}$ and $M_{G_{B_2}}$ are similar matrices.   Consequently, $M_{G_{B_1}}$ and $M_{G_{B_2}}$ have the same characteristic  and minimal polynomial. Note also that $Q^{-1}=Q^{t}$. 
\end{remark}

Let $G_1=(E_{G_1}^0,E_{G_1}^1,r_{G_1},s_{G_1})$ and $G_2=(E_{G_2}^0,E_{G_2}^1,r_{G_2},s_{G_2})$ be two directed graphs. We recall that the \textit{categorical product of $G_1$ and $G_2$} is the directed graph defined by $G_1 \times G_2 := (E_{G_1}^0 \times E_{G_2}^0, E_{G_1}^1 \times E_{G_2}^1, r, s)$ where $s(f,g) =(s(f), s(g))$ and $r(f,g)=(r(f), r(g))$ for any $(f,g)\in E_{G_1}^1 \times E_{G_2}^1$.

\begin{proposition}\label{prop:tensorialproduct}
Suppose $A_1, A_2$ evolution algebras with basis $B_1,B_2$ and directed associated graphs $G_{1},G_{2}$. Then the graph associated to the tensor evolution algebra, $G_{B_1\otimes B_2}$, coincides with the categorical product of $G_{1}$ and $G_{2}$, $G_1\times G_2$. Moreover, $M_{G_1\times G_2}=M_{G_1}\otimes M_{G_2}$.
\end{proposition}

\begin{proof}
Suppose $M_{G_1}=(c_{ik})$ the adjacency matrix of $G_{1}$ ($e_i^2=\sum c_{ik}e_k$), and $M_{G_2}=(d_{jl})$ the adjacency matrix of $G_{2}$ ($f_j^2=\sum d_{jl}f_l$). In order to construct the adjacency matrix of $G_1\times G_2$ we compute $(e_i\otimes f_j)^2=e_i^2\otimes f_j^2=\sum c_{ik} e_k\otimes \sum d_{jl}f_l=\sum (c_{ik} d_{jl})e_k\otimes f_l$. 

We know that: 
\begin{enumerate}
    \item $c_{ik}\neq 0$ if and only if there exists an edge from $e_i$ to $e_k$
    \item $d_{jl}\neq 0$ if and only if there exists an edge from $f_j$ to $f_l$
    \item $c_{ik} d_{jl}\neq 0$ if an only if $c_{ik}\neq 0$ and $d_{jl}\neq 0$
\end{enumerate}
By the definition of the categorical product of graphs this means that there exists an edge from $e_i\otimes f_j$ to $e_k\otimes f_l$. But, these edges are the same that we obtain computing $M_{G_1}\otimes M_{G_2}$. Hence $M_{G_1\times G_2}=M_{G_1}\otimes M_{G_2}$ and therefore $G_1\times G_2=G_{B_1\otimes B_2}$.
\end{proof}

The following result appears in \cite{McANDREW} and concerns strongly connected graphs involving the categorical product of graphs.  

\begin{theorem}[\cite{McANDREW}, Theorem 1.(ii)]\label{thm:strong}
Let $G_1$ and $G_2$ be strongly connected graphs. Let $$d_1=d(G_1)=gcd\{\text{\rm length of all the closed paths in } G_1\},$$
$$d_2=d(G_2)=gcd\{\text{\rm length of all the closed paths in } G_2\},$$
$$d_3=gcd\{d_1,d_2\}.$$
Then the number of connected components of $G_1\times G_2$ is $d_3$.
\end{theorem}

\begin{corollary}\label{cor:strong}
Let $A_i$ ($i=1,2$) be evolution algebras whose associated graphs are $G_i$ ($i=1,2$) respectively. Then we have:
\begin{enumerate}
\item If $G_1$ and $G_2$ are strongly connected graphs and have closed paths of coprime length, then $A_1\otimes A_2$ is nondegenerate and its associated graph is connected, so  $A_1\otimes A_2$ is irreducible.
\item If $A_1\otimes A_2$ is nondegenerate and irreducible, then each factor is nondegenerate and irreducible.
\end{enumerate}
\end{corollary}
\begin{proof}

Since, $G_1$ and $G_2$ are strongly connected graphs,  $d_3=1$. Hence, the categorical product $G_1\times G_2$ has only one connected component, so it is connected.
It is clear that  the evolution algebra associated to a strongly connected graph is nondegenerate. So, $A_1$ and $A_2$ are nondegenerate evolution algebras. By Proposition \ref{rem:zuru}, $A_1\otimes A_2$ is nondegenerate
and by Proposition \ref{prop:tensorialproduct}, its associated graph is connected. 
But this  is equivalent to say that $A_1\otimes A_2$ is irreducible (see \cite[Proposition 2.10]{ELduqueGraphs}). 
For the second part of the Corollary we know that each factor is nondegenerate by Proposition \ref{rem:zuru}. Moreover, since the product graph is connected, then each factor graph has to be connected whence the corresponding algebra is irreducible by \cite[Proposition 2.10]{ELduqueGraphs}.
\end{proof}

\section{Perfect and simple tensor product of evolution algebras}

In this section we continue researching how some properties 
pass from the factors to the tensor product and conversely from the tensor product to the factors under suitable conditions. For instance, when the tensor product is an evolution algebra, then, under certain conditions on one of the factors, the other is an evolution algebra.

\begin{proposition}\label{proposition:uf}
Let $A_1$ and $A_2$ be $\K$-algebras and assume that $A_1\otimes A_2$ is an evolution algebra and $A_2$ has an ideal $J$ of codimension $1$. Then either $A_1$ is an evolution algebra or $A_2^2\subset J$.
\end{proposition}
\begin{proof}
Assume that $A_2^2\not\subset J$, then  there is an algebra epimorphism
$\phi\colon A_2\to\K\cong A_2/J$. Take $b_0\in A_2$ such that $\phi(b_0)=1$.
Define the $\K$-algebra homomorphism $\Omega\colon A_1\otimes A_2\to A_1$ such that
$\Omega(a\otimes b)=\phi(b)a$. For any $a\in A_1$ we have $a=\Omega(a\otimes b_0)$
hence $\Omega$ is an epimorphism. But any epimorphic image of an evolution algebra is an evolution algebra whence $A_1$ is an evolution algebra. Note that when $A_2^2\subset J$, then $\Omega$ is not an algebra homomorphism.
\end{proof}

\begin{remark}\rm\label{remark:application}
To illustrate some application of Proposition \ref{proposition:uf} and 
Proposition \ref{perfect}, assume
that $\dim(A_1)=\dim(A_2)=2$ and $A_1\otimes A_2$ is an evolution
algebra. Then we have the following possibilities:
\begin{enumerate}
\item Both $A_1$ and $A_2$ are simple. 
\item One of them, say $A_2$ has a one-dimensional ideal $J$. If $A_2^2\subset J$
then $A_2$ is not perfect hence $A_1\otimes A_2$ is not perfect by Proposition \ref{perfect}. 
\item One of them, say $A_2$ has a one-dimensional ideal $J$ but $A_2^2\not\subset J$. In this case $A_1$ is an evolution algebra by Proposition \ref{proposition:uf}.
\end{enumerate}
\end{remark}
As a corollary we have 
\begin{corollary}
If the tensor product $A_1\otimes A_2$ of the $\K$-algebras $A_1$ and $A_2$ is a
perfect evolution algebra and $A_2$ has an ideal of codimension $1$, then $A_1$
is an evolution algebra. In particular if $\dim(A_1)=\dim(A_2)=2$ we conclude that
when $A_1\otimes A_2$ is a perfect evolution then either $A_1$ and $A_2$ are simple or
one of them is an evolution algebra.
\end{corollary}
\begin{proof}
The first part of the corollary is a consequence of Proposition \ref{perfect}  and Proposition \ref{proposition:uf}. For the second, if $A_1$ or $A_2$ is not a simple
algebra, without loss of generality we assume that $A_2$ is not simple.
Then $A_2$ has a one-dimensional ideal $J$. If $A_2^2\subset J$, then 
$A_2$ is not perfect implying that $A_1\otimes A_2$ is not perfect by Remark \ref{remark:application} case $(2)$,
a contradiction. Hence $A_2^2\not\subset J$ which implies that $A_1$ is an 
evolution algebra by Remark \ref{remark:application} case $(3)$.
\end{proof}

Consequently, in our quest of tensorially decomposable evolution algebras of dimension $4$, if we focus first on the perfect ones, we know that both factor must be simple, or one of them must be an evolution algebra. So the above results control the amount of cases to be considered.

Now we focus on $2-$dimensional evolution algebras, $A_1$ and $A_2$. First, we prove that not every perfect evolution algebra of dimension $4$ comes from the tensor product of two perfect evolution algebras of dimension $2$ given a concrete example. We need several formulas concerning the number of zeros in the structure matrices.

\begin{remark}\label{remark:zeros}\rm
If $A$ is an $n\times n$ matrix with coefficients in $\K$ we will denote by $z(A)$ the number of zeros in $A$, by  $z_d(A)$ the number of zeros in the diagonal of $A$, by $z_c(A)$ the number of columns filled with zeros of $A$ by $z_r(A)$ the number of rows filled with zeros of $A$ and by $\ran(A)$ the rank of the matrix $A$.
Then if $A$ is $n\times n$ and $B$ is $m\times m$ (both with coefficients in $\K$) we have
\begin{eqnarray}\label{compruebese}
z(A\otimes B)=z(A)m^2+(n^2-z(A))z(B), \label{compruebese1}\\
z_d(A\otimes B)=z_d(A)m+(n-z_d(A))z_d(B),\label{compruebese2}\\
z_c(A\otimes B)=z_c(A)m+(n-z_c(A))z_c(B), 
\label{compruebese3} \\
z_r(A\otimes B)=z_r(A)m+(n-z_r(A))z_r(B). \label{compruebese4}
\end{eqnarray}

Observe that \eqref{compruebese4} is consequence of \eqref{compruebese3} and $(A\otimes B)^t=A^t \otimes B^t$. Now, if $A$ and $B$ are $2 \times 2$ matrices, then $A \otimes B$ cannot have exactly one column with zero entries because the equation $2x + (2-x)y=1$ has no solution in natural numbers. Moreover, since $\ran(A\otimes B)=\ran(A)\ran(B)$ we have that it is not possible that $A \otimes B$ has rank $3$.
In addition, calling $\alpha=z(A\otimes B)$, $k=z_d(A\otimes B)$, $x=z(A)$, $y=z(B)$, $x'=z_d(A)$, $y'=z_d(B)$, we have these possibilities:
\begin{table}[H]\scriptsize
\begin{center}
\begin{tabular}{|c | l |  l|}
\hline
$\boldsymbol{z(A\otimes B)=\alpha}$ & $\boldsymbol{(z(A),z(B))=(x,y)}$ & $\boldsymbol{(z_d(A),z_d(B))=(x',y')\longrightarrow z_d(A\otimes B)=k}$  \\
\hline
$\alpha=0$ & $(x,y)=(0,0)$   & \begin{tabular}{l} 
$(x',y')=(0,0)$  $\longrightarrow k=0$
 \end{tabular}\\
\hline
$\alpha=4$ & $(x,y)=(0,1)$ &
\begin{tabular}{l} 
$(x',y')=(0,0) \longrightarrow k=0 $\\
 $(x',y')=(0,1) \longrightarrow k=2$
 \end{tabular}\\
\hline
$\alpha=7$ & $(x,y)=(1,1)$ & 
\begin{tabular}{l} 
$(x',y')=(0,0) \longrightarrow k=0 $\\
 $(x',y')=(0,1) \longrightarrow k=2$\\
 $(x',y')=(1,1) \longrightarrow k=3$ 
 \end{tabular}\\
\hline
$\alpha=8$ & $(x,y)=(0,2)$  & 
\begin{tabular}{l} 
$(x',y')=(0,0) \longrightarrow k=0$ \\
 $(x',y')=(0,1) \longrightarrow k=2$\\
 $(x',y')=(0,2) \longrightarrow k=4$ 
 \end{tabular}\\
\hline
$\alpha=10$ & $(x,y)=(1,2)$ & 
\begin{tabular}{l} 
$(x',y')=(0,0) \longrightarrow k=0$ \\
 $(x',y')=(0,1) \longrightarrow k=2$\\
 $(x',y')=(1,1) \longrightarrow k=3$\\
 $(x',y')=(0,2),(1,2) \longrightarrow k=4$\\
 \end{tabular}\\
\hline
$\alpha=12$  & $(x,y)=(0,3)$  & 
\begin{tabular}{l} 
 $(x',y')=(0,1) \longrightarrow k=2$\\
 $(x',y')=(0,2) \longrightarrow k=4$
 \end{tabular}\\
\hline
$\alpha=12$ & $(x,y)=(2,2)$  & 
\begin{tabular}{l}
$(x',y')=(0,0) \longrightarrow k=0$\\
$(x',y')=(0,1) \longrightarrow k=2$\\
$(x',y')=(1,1) \longrightarrow k=3$\\
$(x',y')=(0,2),(1,2),(2,2) \longrightarrow k=4$\\
\end{tabular}\\
\hline
$\alpha=13$ & $(x,y)=(1,3)$ & 
\begin{tabular}{l} 
$(x',y')=(0,1) \longrightarrow k=2$ \\
 $(x',y')=(1,1) \longrightarrow k=3$\\
 $(x',y')=(0,2),(1,2) \longrightarrow k=4$\end{tabular}\\
 \hline
$\alpha=14$& $(x,y)=(2,3)$ & 
\begin{tabular}{l} 
$(x',y')=(0,1) \longrightarrow k=2 $\\
 $(x',y')=(1,1) \longrightarrow k=3$\\
 $(x',y')=(0,2), (1,2), (2,2) \longrightarrow k=4$\\
 \end{tabular}\\
 \hline
$\alpha=15$& $(x,y)=(3,3)$ & 
\begin{tabular}{l} $(x',y')=(1,1) \longrightarrow k=3$ \\
 $(x',y')=(1,2),(2,2) \longrightarrow k=4$\\
\end{tabular}\\
\hline
\end{tabular}  
\caption{\footnotesize All possible solutions for \eqref{compruebese2} depending on the number of zeros}
\label{tab:otrazeros2}
\end{center}
\end{table}

In conclusion, a $4 \times 4$ matrix $M$ is tensorially indecomposable if $z(M)\in\{1,2,3,5,6,9,11\}$ or $\ran(M)=3$ or $z_c(M)=1$ or $z_r(M)=1$ or $z_d(M)=1$.

\end{remark}

\begin{remark}\rm 
If $A$ is a $n$-dimensional perfect evolution algebra with structure matrix $M$  then, on the one hand we have that by \cite[Proposition 2.13]{CKS} the unique change of basis matrices are induced by the elements of the semidirect product $S_n \rtimes (\K ^\times)^n $. We define an action of this product on the set of the  matrices of order $n$ (see \cite[section $4.1$]{YolandaTesis}). In this context we can speak of the orbit of an structure matrix. On the other hand, also by \cite[Proposition 2.13]{CKS}, we know that $z(M)$, $z_d(M)$ and $\ran(M)$ are invariants for all the elements on the orbit. Moreover, $z_c(M)$ (dimension of annihilator) and $z_r(M)$ are also invariants. Furthermore, if we consider the number of zero entries in a chosen column $i$ we define $z_c^i(M)$ as the number of zero entries in the $i^{th}$ column and similarly we define $z^j_r(M)$ as the number of zero entries in the $j^{th}$ row. Then, for every natural basis $B'$ there exists a permutation $\sigma \in S_n$ such that $z_c^i(M_B)=z_c^{\sigma(i)}(M_{B'})$. Note that $M_{B'}=(M_\sigma P)^{-1}M_B M_\sigma P^2$ where $M_\sigma$ is the permutation matrix relative to $\sigma$ and $P$ is a invertible diagonal matrix. So, the previous remark gives us a method to discard all the matrices on an orbit that do not come from the Kronecker product of two matrices by counting the number of zeros or computing the rank of one of them.
\end{remark}

\begin{example}\rm 
If $A$ is the perfect evolution algebra with structure matrix 
$M=\tiny\begin{pmatrix}
1 & 1 & 1 & 1\\
0 & 1 & 1 & 1\\
0 & 0 & 1 & 1\\
0 & 0 & 0 & 1
\end{pmatrix}$, then $A$ is tensorially indecomposable since $z(M)=6$ (by using the conclusion given in Remark \ref{remark:zeros}).
\end{example}

\begin{example}\rm \label{ej:ceros}
Now, consider the perfect evolution algebra $A$ with structure matrix
$N=\tiny\begin{pmatrix}
1 & 1 & 1 & 1\\
0 & 1 & 1 & 1\\
0 & 0 & 1 & 0\\
0 & 0 & 0 & 1
\end{pmatrix}$.
Following Table \ref{tab:otrazeros2} we may be in one of the cases, since $z(N)=7$ and $z_d(N)=0$, nevertheless $N$ is tensorially indecomposable as we will see in Example \ref{ex:ext}.

\end{example}

\begin{remark}\label{remark:extendida}\rm
Now, we describe another well-known tool to check if a matrix is tensorially decomposable. 
We define the map $\omega: \mathcal{M}_m(\K)\to \K^{m^2}$ such that, if  $M=(a_{ij})_{i,j=1}^{m}$, then\\ $$\omega(M)=(a_{11},a_{12},\ldots,a_{1m}, a_{21},\ldots,a_{2m},\ldots,a_{mm}).$$
Consider a general matrix $M$ that comes from the Kronecker product of two matrices of order $n$ and $k$ respectively ($n,k > 1$). That is,
$M=\tiny\begin{pmatrix}
V_{11} & \ldots & V_{1k}\\
\vdots & \ddots & \vdots\\
V_{k1} & \ldots & V_{kk}
\end{pmatrix}$
with $V_{ij}\in \mathcal{M}_n(\K)$ and $M\in \mathcal{M}_{nk}(\K)$.
We now construct a matrix whose rows are the extended blocks that appear in the Kronecker product matrix. Let us call it the \textit{extended matrix} of $M$ and denote by $\ex(M)$ the following matrix
\begin{equation}\label{extended}
\ex(M)=\tiny\begin{pmatrix}
\omega(V_{11})\\
\omega(V_{12})\\
\vdots\\
\omega(V_{1k})\\
\vdots\\
\omega(V_{kk})
\end{pmatrix}.
\end{equation}
Observe that a matrix $N$ with order $nk$ and $n, k >1$ is tensorially decomposable into two matrices of order $n$ and $k$ respectively if and only if $\ran(\ex(N))=1$.

Let us consider a finite dimensional perfect evolution algebra with structure matrix $M_B$ relative to a natural basis $B$. It is easy to check that there exist examples such that $\ran(\ex (M_B))=1$ but $\ran(\ex (M_{B'})) \neq 1$ where $B'$ is another natural basis.

\end{remark}

\begin{example}\label{ex:ext}\rm 
If we apply Remark \ref{remark:extendida} to the structure matrix $N$ that appears in Example \ref{ej:ceros} one can check that the rank of all the extended matrices in the orbit of $N$ is $3$. Hence, $N$ does not come from the Kronecker product of two $2\times 2$ matrices.
\end{example}

Let us give all the possible tensor product of two $2-$dimensional perfect evolution algebras regarding the number of zeros in the corresponding structure matrices $M_1$ and $M_2$. It is known that the perfect simple evolution algebras of dimension $2$ have one of the following structure matrix:
$\tiny\begin{pmatrix}
1 & a\\
b & 1
\end{pmatrix}$ with $1-ab \neq 0$, $\tiny\begin{pmatrix}
0 & c\\
1 & 0
\end{pmatrix}$ with $c \neq 0$ and
$\tiny\begin{pmatrix}
0 & 1\\
d & 1
\end{pmatrix}$ with $d \neq 0$ and $a, b, c, d \in \K$ (see \cite[Lemma 3.1]{squares}).

Now, attending to these types of possible structure matrices  for $2-$dimensional perfect evolution algebras, we can distribute the tensor product $M=M_1\otimes M_2$  according to   Table \ref{tab:productos} to distinguish the different family types and  we will look for a complete system of invariants that will be formed by some of the following elements: $z(M)$, $z_d(M)$, simplicity of the tensor product, $p_c(M)$ and $p_m(M)$, where $p_c(M)$ and $p_m(M)$ denote the characteristic and minimal polynomials of $M$ respectively.

We have applied the $\sigma_k$-criterion to determine the case in which $A_1\otimes A_2$ is simple: it turns out that they are  family types I-V, VII, XI appearing in Table \ref{tab:productos}.

\begin{table}\tiny
\begin{center}
\setlength\tabcolsep{1.5pt} 
\begin{tabular}{|c||c|c | c |  c|| c|c | c|c|}
\hline
$\begin{matrix}
\textbf{Family}\\
\textbf{Type}
\end{matrix}$&$\boldsymbol{z(M_1)}$ & $\boldsymbol{z(M_2)}$ & $\boldsymbol{M_1\otimes M_2=M}$ & \textbf{Conditions} & \textbf{Simple} & $\boldsymbol{z(M)}$ &  $\begin{cases}\boldsymbol{p_c(M)}\\ \boldsymbol{p_m(M)} \end{cases}$ \\
\hline
I&0 & 0 & $\tiny\begin{pmatrix}
1 & a\\ b&1
\end{pmatrix}\otimes \tiny\begin{pmatrix}
1 & c\\ d&1
\end{pmatrix}=\tiny\begin{pmatrix}
1 & c & a & ac\\
d & 1 & ad & a\\
b & bc & 1 & c\\
bd & b & d & 1
\end{pmatrix}$ & $\begin{matrix}
ab\neq 1, cd\neq 1\\
abcd\neq 0
\end{matrix}$  & Yes & 0 & $\begin{cases} \lambda^3(\lambda-4)\\ \lambda(\lambda-4)\end{cases}$\\
\hline
II&1 & 0 & $\tiny\begin{pmatrix}
0 & 1\\ b & 1
\end{pmatrix}\otimes \tiny\begin{pmatrix}
1 & c\\ d & 1
\end{pmatrix}=\tiny\begin{pmatrix}
0 & 0 & 1 & c\\
0 & 0 & d & 1\\
b & bc & 1 & c\\
bd & b & d & 1
\end{pmatrix}$ & $\begin{matrix}bcd\neq 0, cd\neq 1  \end{matrix}$
& Yes & 4 &$\begin{cases} \lambda^2 (\lambda^2 -2 \lambda -4)\\ \lambda(\lambda^2-2\lambda-4)\end{cases}$\\
\hline
III&1 & 0 & $\tiny\begin{pmatrix}
1 & 0\\ b & 1
\end{pmatrix}\otimes \tiny\begin{pmatrix}
1 & c\\ d & 1
\end{pmatrix}=\tiny\begin{pmatrix}
1 & c & 0 & 0\\
d & 1 & 0 & 0\\
b & bc & 1 & c\\
bd & b & d & 1
\end{pmatrix}$ &  $\begin{matrix}cd\neq 1, bcd\ne0 \end{matrix}$& Yes & 4 & $\begin{cases} \lambda ^2 (\lambda-2)^2\\ \lambda(\lambda-2)^2\end{cases}$\\
\hline
IV&1 & 1 & $\tiny\begin{pmatrix}
0 & 1\\ b & 1
\end{pmatrix}\otimes \tiny\begin{pmatrix}
0 & 1\\ d & 1
\end{pmatrix}=\tiny\begin{pmatrix}
0 & 0 & 0 & 1\\
0 & 0 & d & 1\\
0 & b & 0 & 1\\
bd & b & d & 1
\end{pmatrix}$ &   $bd\ne0$ & Yes & 7 & $\begin{cases}(\lambda+1)^2(\lambda^2-3\lambda+1)\\(\lambda+1)(\lambda^2-3\lambda+1)\end{cases}$  \\
\hline
V&1 & 1 & $\tiny\begin{pmatrix}
0 & 1\\ b & 1
\end{pmatrix}\otimes \tiny\begin{pmatrix}
1 & 0\\ d & 1
\end{pmatrix}=\tiny\begin{pmatrix}
0 & 0 & 1 & 0\\
0 & 0 & d & a\\
b & 0 & 1 & 0\\
bd & b & d & 1
\end{pmatrix}$ &  $bd\ne0$ & Yes & 7 & $\begin{cases} (\lambda^2-\lambda -1)^2\\  (\lambda^2-\lambda -1)^2 \end{cases}$\\
\hline
VI&1 & 1 & $\tiny\begin{pmatrix}
1 & 0\\ b & 1
\end{pmatrix}\otimes \tiny\begin{pmatrix}
1 & 0\\ d & 1
\end{pmatrix}=\tiny\begin{pmatrix}
1 & 0 & 0 & 0\\
d & 1 & 0 & 0\\
b & 0 & 1 & 0\\
bd & b & d & 1
\end{pmatrix}$ &   $bd\ne0$ & No & 7 & $\begin{cases}(\lambda-1)^4\\ (\lambda-1)^3\end{cases}$\\
\hline
VII&2 & 0 & $\tiny\begin{pmatrix}
0 & a\\ 1 & 0
\end{pmatrix}\otimes \tiny\begin{pmatrix}
1 & c\\ d & 1
\end{pmatrix}=\tiny\begin{pmatrix}
0 & 0 & a & ac\\
0 & 0 & ad & a\\
1 & c & 0 & 0\\
d & 1 & 0 & 0
\end{pmatrix}$ &$\begin{matrix}
acd\neq 0, cd\neq 1\\
\end{matrix}$  & Yes & 8 & $\begin{cases}\lambda^2(\lambda+2)(\lambda-2)\\ \lambda(\lambda+2)(\lambda-2)\end{cases}$\\
\hline
VIII&2 & 0 & $\tiny\begin{pmatrix}
1 & 0\\ 0 & 1
\end{pmatrix}\otimes \tiny\begin{pmatrix}
1 & c\\ d&1
\end{pmatrix}=\tiny\begin{pmatrix}
1 & c & 0 & 0\\
d & 1 & 0 & 0\\
0 & 0 & 1 & c\\
0 & 0 & d & 1
\end{pmatrix}$ &   $cd\ne 0, cd\ne1$ & No & 8 & $\begin{cases}
\lambda^2(\lambda-2)^2\\
\lambda(\lambda-2)
\end{cases}$ \\
\hline
IX&2 & 1 & $\tiny\begin{pmatrix}
1 & 0\\ 0&1
\end{pmatrix}\otimes \tiny\begin{pmatrix}
0 & 1\\ d & 1
\end{pmatrix}=\tiny\begin{pmatrix}
0 & 1 & 0 & 0\\
d & 1 & 0 & 0\\
0 & 0 & 0 & 1\\
0 & 0 & d & 1
\end{pmatrix}$ & $d\ne0$& No & 10 &  $\begin{cases} (\lambda^2-\lambda-1)^2\\  (\lambda^2-\lambda-1)\end{cases}$\\
\hline
X&2 & 1 & $\tiny\begin{pmatrix}
1 & 0\\ 0&1
\end{pmatrix}\otimes \tiny\begin{pmatrix}
1 & 0\\ d & 1
\end{pmatrix}=\tiny\begin{pmatrix}
1 & 0 & 0 & 0\\
d & 1 & 0 & 0\\
0 & 0 & 1 & 0\\
0 & 0 & d & 1
\end{pmatrix}$ & $ d\neq 0$ & No & 10 &$\begin{cases} (\lambda -1)^4\\(\lambda -1)^2\end{cases}$\\
\hline
XI&2 & 1 & $\tiny\begin{pmatrix}
0 & a\\ 1&0
\end{pmatrix}\otimes \tiny\begin{pmatrix}
0 & 1\\ d & 1
\end{pmatrix}=\tiny\begin{pmatrix}
0 & 0 & 0 & a\\
0 & 0 & ad & a\\
0 & 1 & 0 & 0\\
d & 1 & 0 & 0
\end{pmatrix}$ & $ ad\ne 0$ & Yes & 10 &$\begin{cases} (\lambda^2-\lambda-1)(\lambda^2+\lambda-1)\\ (\lambda^2-\lambda-1)(\lambda^2+\lambda-1)\end{cases}$\\
\hline
XII&2 & 1 & $\tiny\begin{pmatrix}
0 & a\\ 1&0
\end{pmatrix}\otimes \tiny\begin{pmatrix}
1 & 0\\ d & 1
\end{pmatrix}=\tiny\begin{pmatrix}
0 & 0 & a & 0\\
0 & 0 & ad & a\\
1 & 0 & 0 & 0\\
d & 1 & 0 & 0
\end{pmatrix}$ & $ a d\ne 0$ & No & 10 & $\begin{cases}(\lambda-1)^2(\lambda +1)^2\\(\lambda-1)^2(\lambda +1)^2 \end{cases}$\\
\hline
XIII&2 & 2 & $\tiny\begin{pmatrix}
1 & 0\\ 0 & 1
\end{pmatrix}\otimes \tiny\begin{pmatrix}
0 & 1\\ d & 0
\end{pmatrix}=\tiny\begin{pmatrix}
0 & 0 & 0 & 1\\
0 & 0 & d & 0\\
0 & 1 & 0 & 0\\
d & 0 & 0 & 0
\end{pmatrix}$ &  $d\ne0$ & No &  12 & $\begin{cases}(\lambda-1)^2(\lambda +1)^2\\ (\lambda-1)(\lambda +1)\end{cases}$\\
\hline
\end{tabular}  
\caption{\footnotesize Tensorially decomposable $4-$dimensional evolution algebra types}
\label{tab:productos}
\end{center}
\end{table}

Observe that we have obtained several complete systems of invariants for all family types in Table \ref{tab:productos}. Then, a complete system of invariants consist on each of the sets composed by:
\begin{itemize}
    \item The numbers $z(M)$, $z_d(M)$ and the simplicity of the tensor product.
    \item The number $z(M)$ and the polynomial $p_c(M)$.
    \item The polynomials $p_c(M)$ and $p_m(M)$.
\end{itemize}

\begin{remark}\rm 
If $\{e_i\}$ and $\{f_j\}$ are natural bases of a perfect evolution algebra, then
we know that there is a permutation $\sigma$ and nonzero scalars $c_j$ such that $f_j=c_j e_{\sigma(j)}$ for any $j$. Assuming that $e_i^2=\sum_j \omega_{ji}e_j$ and $f_j^2=\sum_{k}\sigma_{kj}f_k$ (for $\omega_{ji},\sigma_{kj}\in\K$) we know that the relation between the structure constants is 
$$\sigma_{ij}=c_j^2c_i^{-1}\omega_{\sigma(i),\sigma(j)}.$$ 
In particular, if the structure constants satisfy $\omega_{ii}=\sigma_{ii}=1$ for any $i$, we have $c_i=1$ also for any $i$. Accordingly one basis is a re-ordering of the other. In terms of the basis change matrix, it reduces to a permutation matrix.
\end{remark}

Now we want to classify all the tensorially decomposable, $A_1\otimes A_2$, evolution algebras of dimension $4$ distinguishing two cases: $A_1\otimes A_2$ simple and $A_1\otimes A_2$ perfect (but not simple).

\begin{theorem}\label{clasifica}
Let $A_1$ and $A_2$ be $2$-dimensional perfect evolution algebras. We have:
\begin{itemize}
    \item If $A_1\otimes A_2$ is a $4-$dimensional  simple evolution algebra, then it turns out that they are one of the family types I-V, VII and XI.
    \item If $A_1\otimes A_2$ is a $4-$dimensional non simple evolution algebra, then it turns out that they are one of the family types VI, VIII-X, XII, XIII.
\end{itemize}
\end{theorem}
In the case $A_1\otimes A_2$ is a non perfect evolution algebra then it must be that either $A_1$ or $A_2$ is a non perfect evolution algebra by Remark \ref{nonperfect}. 
The classification of  tensorially decomposable $4-$dimensional non perfect evolution algebras will be addressed in a forthcoming work. 
\begin{example}\rm
As an example of a classification task of evolution algebras assume the evolution algebra with structure matrix
$$M =\tiny \left(
\begin{array}{cccc}
 1 & 0 & 0 & 2 \\
 0 & 1 & 0 & 1 \\
 1 & 2 & 1 & 2 \\
 0 & 0 & 0 & 1 \\
\end{array}
\right).$$
Since the determinant of $M$ is nonzero, this matrix corresponds to a perfect evolution algebra $A$. We want to classify it. If we compute the stabilizing index of $M$ 
we see that it is $1$ since $\sigma_2(M)=\sigma_1(M)$.
However, $\sigma_2(M)$ is not the matrix whose entries are $1$, consequently $A$ is not a simple algebra.
Next, we can compute the characteristic polinomial of $$\tiny \left(
\begin{array}{cccc}
 1 & 0 & 0 & 1 \\
 0 & 1 & 0 & 1 \\
 1 & 1 & 1 & 1 \\
 0 & 0 & 0 & 1 \\
\end{array}
\right)$$
and we find $(\lambda-1)^4$ hence if $A$ is tensorially decomposable it should be an algebra of the family type VI or X. 
Now the minimal polinomial of the above matrix is $(\lambda-1)^3$ hence this determines that the only possible family type is VI.
So, in order to see if it is tensorially decomposable we compute the orbit of $M$ under the action of $S_4$ and check if there is some element $M'$ in that orbit, such that $\ex(M')=1$ (see Remark \ref{remark:extendida}).
Computing the orbit of $M$ we find $$M'=\tiny\left(
\begin{array}{cccc}
 1 & 0 & 0 & 0 \\
 2 & 1 & 0 & 0 \\
 1 & 0 & 1 & 0 \\
 2 & 1 & 2 & 1 \\
\end{array}
\right)$$ which satisfies $\ex(M')=1$ and in fact $M'=\tiny\begin{pmatrix}1 & 0\cr 1 & 1\end{pmatrix}\otimes 
\begin{pmatrix}1 & 0\cr 2 & 1\end{pmatrix}$.
\end{example}

\end{document}